\documentclass[11pt, a4paper, oneside]{amsart}

\usepackage[text={5.13in,8.5in}, top=1.3in]{geometry}

\usepackage{amsmath}
\usepackage{amstext}
\usepackage{amssymb}

\usepackage{enumerate}
\usepackage{enumitem}

\usepackage[T1]{fontenc}

\usepackage{tikz}

\usepackage[applemac]{inputenc}

\usepackage{hyperref}

\makeatletter
\def\imod#1{\allowbreak\mkern10mu({\operator@font mod}\,\,#1)}
\makeatother

\newtheorem{theorem}{Theorem}[section]

\newtheorem{lemma}[theorem]{Lemma}
\newtheorem{claim}[theorem]{Claim}
\newtheorem{corollary}[theorem]{Corollary}

\theoremstyle{definition}
\newtheorem{definition}[theorem]{Definition}

\newtheorem{fact}[theorem]{Fact}

\newtheorem{question}[theorem]{Question}
\theoremstyle{remark}
\newtheorem{remark}[theorem]{Remark}

\theoremstyle{remark}

\newtheorem{convention}[theorem]{Convention}

\numberwithin{equation}{section}

    \DeclareMathOperator{\dom}{dom}
    
    \DeclareMathOperator{\ran}{ran}

    \DeclareMathOperator{\fix}{fix}
    \DeclareMathOperator{\cofin}{cofin}

    \DeclareMathOperator{\eL}{\mathbf{L}}
    \DeclareMathOperator{\Vee}{\mathbf{V}}

    \newcommand{\forces}{\Vdash}

\def\C{{\mathbb C}}

\def\N{{\mathbb N}}

\def\bF{{\mathbb F}}
\def\P{{\mathbb P}}

\def\Q{{\mathbb Q}}

\DeclareMathOperator{\one}{id_\N}

\DeclareMathOperator{\mpath}{path}
\DeclareMathOperator{\set}{use}

\DeclareMathOperator{\parfun}{par}
\DeclareMathOperator{\parfin}{fin}

\title{A co-analytic Cohen-indestructible maximal cofinitary group}

\author{Vera Fischer}
\address{Vera Fischer, Kurt G\"odel Research Center, Universit\"at Wien, W\"ahringerstra\ss e 25, 1080 Wien, Austria}
\email{vera.fischer@univie.ac.at}

\author{David Schrittesser}
\address{David Schrittesser, Department of Mathematical Sciences, University of Copenhagen, Universitetsparken 5, 2100 Copenhagen, Denmark}
\email{david.s@math.ku.dk}

\author{Asger T\"ornquist}
\address{Asger T\"ornquist, Department of Mathematical Sciences, University of Copenhagen, Universitetsparken 5, 2100 Copenhagen, Denmark}
\email{asgert@math.ku.dk}



\begin{document}

\maketitle

\begin{abstract}
Assuming that every set is constructible, we find a $\Pi^1_1$ maximal cofinitary group of permutations of $\N$ which is indestructible by Cohen forcing.
Thus we show that the existence of such groups is consistent with arbitrarily large continuum.
Our method also gives a new proof, inspired by the forcing method, of Kastermans' result that there exists a $\Pi^1_1$ maximal cofinitary group in $\eL$.
\end{abstract}

\section{Introduction}\label{intro}

{\bf (A)} We denote  the group of permutations (bijections) of $\N$  by $S_\infty$, and its unit element by $\one$.
An element of $S_\infty$ is cofinitary if and only if it has only finitely many fixed points, 
and  $\mathcal{G}$  is called a \emph{cofinitary group} precisely if $\mathcal{G} \leq S_\infty$ and all elements of $\mathcal G\setminus\{ \one\}$ are cofinitary.

A cofinitary group is said to be \emph{maximal} if and only if it is maximal under inclusion among cofinitary groups.

\medskip

Various aspects of maximal cofinitary groups (or short, \emph{mcg}s) have long been studied (see e.g.\ \cite{cameron-cofinitary,cameron-aspects,adeleke-embeddings,truss-embeddings,truss-joint,koppelberg-groups,hjorth-cameron}), including
 possible sizes of mcgs; their relation to maximal almost disjoint (or \emph{mad}) families, of which they are examples;
as well as inequalities relating $\mathfrak{a}_g$, i.e.\ the least size of a mcg, to other cardinal invariants of the continuum; see e.g.\ \cite{zhang-maximal,zhang-permutation,hrusak-cofinitary,brendle-uniformity,kastermans-cardinal,fischer-toernquist-template}.
Analogous questions about permutation groups on $\kappa$, where $\kappa$ is an uncountable cardinal, have also been studied; see e.g.\ \cite{fischer-maximal}.
The isomorphism types of mcgs have been investigated in \cite{kastermans-isomorphism}.

The line of research to which this paper belongs concerns the \emph{definability} of mcgs. 

\medskip 

{\bf (B)} Since the existence of mcgs relies on the axiom of choice, the question of whether a mcg can be \emph{definable} has drawn considerable interest.

It was shown by Truss \cite{truss-embeddings} and Adeleke \cite{adeleke-embeddings} that no mcg can be countable; this was improved by Kastermans'  
result \cite[Theorem 10]{kastermans-complexity} that no mcg can be $K_\sigma$.
On the other hand,  Gao and Zhang \cite{gao-definable} showed that assuming $\Vee=\eL$, there is a mcg with a co-analytic generating set.
This, too, was improved by Kastermans with the following theorem.

\begin{theorem}[\cite{kastermans-complexity}]\label{t.kastermans}
If $\Vee=\eL$ there is a $\Pi^1_1$ (i.e.\ effectively co-analytic) mcg.
\end{theorem}
The previous theorem immediately raises the question of whether the existence of a $\Pi^1_1$ mcg is consistent with $\Vee\neq \eL$, or even with the negation of the continuum hypothesis.
In this paper we answer these questions in the positive:
\begin{theorem}\label{t.continuum}
The existence of a $\Pi^1_1$ mcg is consistent with arbitrarily large continuum (assuming the consistency of ZFC).
\end{theorem}

At the same time we give a new proof of Kastermans' Theorem \ref{t.kastermans}.
This is worthwhile for several reasons: Firstly, our method shows that in $\eL$, \emph{any} countable cofinitary group is contained in a co-analytic mcg. Secondly, the `coding technique' which ensures that the group is co-analytic, described in  Definition \ref{d.coding}, is much more straightforward than the one in \cite{kastermans-complexity}. Thirdly, this method seems open to a wider range of variation, allowing to construct mcgs with additional properties.

\medskip

An example of such a property is \emph{Cohen-indestructibility}, which we now define.
For this, first observe that if $\mathcal G$ is a cofinitary group, then clearly it remains so in any extension of the universe.
\begin{definition}\label{d.indestructible}
Let $\mathcal G$ be a mcg and let $\C$ denote Cohen forcing. We say
$\mathcal G$ is \emph{Cohen-indestructible} if and only if 
$\forces_{\C} \check{ \mathcal G}\text{ is maximal}$.
\end{definition}

A Cohen-indestructible mcg was first obtained by Zhang \cite{zhang-permutation}. 
The following is our main result; Theorem \ref{t.continuum} is clearly a corollary.
\begin{theorem}\label{t.indestructible}
If $\Vee=\eL$, there is a $\Pi^1_1$ Cohen-indestructible mcg.
\end{theorem}
To prove the theorem, we first find a forcing which, given a cofinitary group $\mathcal G$ and $z\in2^\N$, adds a generic cofinitary group $\mathcal G'$ such that $\mathcal G\leq \mathcal G'$ 
and with the property that $z$ is computable from (or `is coded by', see Definition \ref{d.coding}) each element from a sufficiently large subset of $\mathcal G'$.
To find this forcing, we refine Zhang's forcing from \cite{zhang-maximal} (also see \cite{fischer-toernquist-template} and \cite{fischer-maximal} for variations). 

We then use this to give a new proof of Kastermans' Theorem \ref{t.kastermans}, building our group from permutations which are generic over certain countable initial segments of $\eL$.
We use ideas from \cite{fischer-maximal} to see that the group produced in this manner is Cohen-indestructible.

\medskip

{\bf (C)} The paper is structured as follows.
In \S \ref{s.notation}, we establish basic terminology for \S \ref{s.Q}. 
In \S \ref{ss.simple} we give a streamlined presentation of Zhang's forcing $\Q_{\mathcal G}$, in order to simplify the definition and discussion of our forcing $\Q^z_{\mathcal G}$, which follows in \S \ref{ss.Q}.
In \S \ref{s.main}, we prove our main result, Theorem \ref{t.indestructible}, in a slightly more general form
(Theorem \ref{t.main}), after a short review of facts of effective descriptive set theory and fine structure theory used in the proof.
We close in \S \ref{s.questions} by listing some questions which remain open.

\medskip

\emph{Acknowledgements}.
The first author would like to thank the
 Austrian Science Fund (FWF) for the generous support through Lise-Meitner
 research grant M1365-N13.
The second and third authors gratefully acknowledge the generous support from Sapere Aude grant no. 10-082689/FNU from Denmark's Natural Sciences Research Council, and from the DNRF Niels Bohr Professorship of Lars Hesselholt. 

\emph{Addendum}. After the appearance of this paper, Horowitz-Shelah \cite{borel-mcg} showed in ZF
that there is a Borel maximal cofinitary group.  Note that our result remains interesting, as the group we construct is still in $\eL$, of size $\omega_1$, and maximal after adding any number of Cohen reals, while the group from \cite{borel-mcg} always has size $2^\omega$.

\section{Notation and Preliminaries}\label{s.notation}

We start by reviewing the necessary definitions and introduce convenient terminology, in particular the notion of a \emph{path}.

\medskip 

{\bf(A)} Since we build a generic element of $S_\infty$ from finite approximations, we shall work with partial functions.
We write $\parfun(\N,\N)$ for the set of 
partial functions from $\N$ to $\N$, and $\parfin(\N,\N)$ for the set of finite such functions.
For $a \in \parfun(\N,\N)$, when we write $a(n)$ it is clearly implied that $n \in \dom(a)$
except when we say
\emph{a(n) is undefined}, which means that $n \notin \dom(a)$.
For the set of \emph{fixed points of $a$} we write
\[
\fix(a)=\{ n \in \N : a(n)=n \}.
\]

The set $\parfun(\N,\N)$ is naturally equipped with the operation of composition of partial functions
\[
(f g) (n) = m \iff f(g(n))=m,
\]
making it an associative monoid.

\medskip

Let $\mathcal{G}$ be an arbitrary group.
By $\bF(X)$ we denote the \emph{free group with single generator $X$}.
We identify the group $\mathcal G * \bF(X)$, i.e.\ the free product of $\mathcal G$ and $\bF(X)$, with the set $W_{\mathcal G,X}$ of reduced words from the alphabet $\big(\mathcal G\setminus\{ 1_{\mathcal G}\}\big)\cup \{ X, X^{-1}\}$, equipped with the familiar `concatenate and reduce' operation (see e.g.\ \cite[Normal Form Theorem]{lyndon-combinatorial}). The neutral element $1_{\mathcal G}$ is therefore identified with the empty word, which we denote by $\emptyset$ .

\medskip

By a \emph{circular shift} of a non-empty word $w=w_n\hdots w_1$ we mean the result of reducing the word
$w_{\sigma(n)} \hdots w_{\sigma(1)}$,
where $\sigma$ is a permutation of $\{ 1, \hdots, n\}$ such that for some  $k\in \N$, 
\[
\sigma(i) = i + k  \bmod n.
\]
Thus,  e.g.\ $cdab$ is a circular shift of $abcd$ (in the free group  generated by $\{a,b,c,d\}$).
By a subword of $w$ we mean a contiguous subword $w_i \hdots w_j$ for $n\geq i\geq j\geq1$, or the empty word.
Of course, the empty word is both the only circular shift and the only subword of itself.
\medskip

We call a group homomorphism $\rho \colon\mathcal G \to S_\infty$
a \emph{cofinitary representation of $\mathcal G$} if and only if $\rho[\mathcal G]$ is cofinitary.
Clearly, if $\rho$ is injective (i.e.\ a \emph{faithful representation}), we may identify $\mathcal G$ with the  cofinitary group $\rho[\mathcal G] \leq S_\infty$.

\medskip

For the remainder of this section assume $\mathcal G \leq S_\infty$.
Choosing an arbitrary $s \in \parfun(\N,\N)$ 
gives rise to a unique homomorphism of monoids
\[
\rho\colon\mathcal G * \bF(X) \rightarrow \parfun(\N,\N)
\]
such that $\rho(X)=s$ and $\rho$ is the identity on $\mathcal G$.
It can be defined by induction on the length of words in the obvious way.
Let's denote this homomorphism by $\rho_{\mathcal G, s}$, departing from \cite{fischer-toernquist-template} (where it is precisely the map $w \mapsto e_w(s)$).
Its image is the compositional closure  $\langle \mathcal G, s \rangle$ of $\mathcal G \cup \{ s\}$ in $\parfun(\N,\N)$.

\begin{convention}\label{convention}
We adopt the convention to denote $\rho_{\mathcal G, s}(w)$ by
$w[s]$, for any $w  \in W_{\mathcal G, X}$ (we `substitute $s$ for $X$ in $w$'; as e.g.\ in \cite{zhang-maximal,kastermans-complexity}).
\end{convention}
Observe that slightly awkwardly by this convention $\emptyset[s] = \one$ for any $s\in\parfun(\N,\N)$.

\medskip

{\bf(B)} We define the notion of a \emph{path}, which will be extremely useful in the next section. 
Fix $s \in \parfun(\N, \N)$.
Say  $w \in W_{\mathcal G, X}$, and in reduced form
\[
w= a_n  \hdots a_1.
\]

We  define the \emph{path under $(w,s)$ of $m$}, also called the \emph{$(w,s)$-path of $m$} and written
\[
\mpath(w,s,m) = \langle m_i \colon i \in \alpha \rangle,
\]
to be the following sequence of natural numbers:
$m_0=m$ and for $l\in\N$ and $i \in \N$ such that $0<i\leq n$, 
\begin{align*}
m_{n l + i}  =& a_i \hdots a_1 w^l[s] (m_0)%
\end{align*}
and $\alpha \in \omega+1$  is maximal such that all of these expressions are defined.
That is we simply iterate applying all the letters of $w$ as they appear from right to left and record the outcome until we reach an undefined expression.

We can represent such a path e.g.\ as follows (where $i = 1 +( k \bmod n)$):
\[
\hdots m_{k+1}\xleftarrow{\ a_i} m_{k}\hdots\stackrel{\ a_2}{\longleftarrow} m_{n+1} \stackrel{\ a_1}{\longleftarrow} m_n \stackrel{\ a_n}{\longleftarrow}  \hdots  \stackrel{\ a_2}{\longleftarrow} m_1 \stackrel{\ a_1}{\longleftarrow} m_0,
\]
or more simply, we shall represent it as $\langle \hdots, m_n, \hdots, m_1, m_0\rangle$.

For $k < \alpha$ and $i  =1 +( k \bmod n)$, we say $a_i$ \emph{occurs} at step $k+1$ in the above path;
if $m_{k+1}$ is defined we also say  $a_i$ \emph{is applied 
(to $m_{k}$)} at this step.
If $m_{k+1}$ is undefined, we say the path \emph{terminates after $\alpha-1=k$ steps with last value $m_k$} or that it \emph{terminates before (an occurrence of) the letter $a_i$}.

Sometimes we are interested in the path merely as a set, rather than as a sequence; so let
\[
\set(w,s,m)=\{ m_i \colon i < \alpha \}.
\]

\medskip

{\bf (C)} 
Of course, we identify $\N$ and $\omega$, but prefer to denote this set as $\N$ in the context of permutations.
We denote by $\lvert A \rvert$ the cardinality of $A$, for any set $A$.
We sometimes, but not always, decorate names in the forcing language with dots and checks, with the goal of aiding the reader. Notation regarding forcing is as in \cite{jech-set}.

\section{Coding into a generic group extension}\label{s.Q}

Fix, for this section, a cofinitary group $\mathcal G \leq S_\infty$. 
We want to enlarge it by $\sigma^* \in S_\infty$, such that $\langle \mathcal G, \sigma^*\rangle$ is cofinitary.
This can be done using a forcing invented by Zhang \cite{zhang-maximal}; for some of its applications see  \cite{brendle-uniformity,zhang-permutation,hrusak-cofinitary,zhang-constructing,kastermans-cardinal,gao-definable,kastermans-analytic}.
 
  \medskip
 
In \S \ref{ss.Q}, we introduce a new forcing $\Q^z_{\mathcal G}$,
such that in addition to the above, every permutation in $\langle \mathcal G,\sigma^*\rangle$ with a certain property `codes' a given, fixed $z \in2^\N$, in a certain sense (see below).

Before we introduce this new forcing notion,  we define our own version of Zhang's forcing,
$\Q_{\mathcal G}$ in \S \ref{ss.simple}, differing slightly from \cite{zhang-maximal}.
We then  analyse carefully how paths behave when conditions in $\Q_{\mathcal G}$ are extended,
facilitating the treatment of $\Q^z_{\mathcal G}$.

 \medskip

Note that in the case of countable $\mathcal G$, Zhang's $\Q_{\mathcal G}$ from \cite{zhang-maximal}, our version of  $\Q_{\mathcal G}$ described in \S \ref{ss.simple}, and the forcing $\Q^z_{\mathcal G}$ are all countable, i.e.\ particular presentations of Cohen forcing.

\subsection{Zhang's forcing revisited}\label{ss.simple}

We now turn to our definition of the forcing to add a generic faithful cofinitary representation of $\mathcal G*\bF(X)$.

\begin{definition}[The forcing $\Q_{\mathcal G}$]\label{d.Q.simple}~ 
\begin{enumerate}[label=(\alph*),ref=\alph*]
\item
Conditions of $\Q_{\mathcal G}$ are pairs $p=(s^p,F^p)$, where $s \in \parfin(\N,\N)$ is injective and $F^p \subseteq W_{\mathcal G, X}$ is finite and closed under taking subwords.
\item  \label{d.Q.simple.order}
$(s^q,F^q)\leq_{\Q_{\mathcal G}} (s^p,F^p)$ if and only if $s^q \supseteq s^p$, $F^q \supseteq F^p$ and for all $w\in F^p\setminus\mathcal G$,
if $m \in \fix(w[s^q])$ then there is a non-empty subword $w'$ of $w$ such that
letting $w=w_1w' w_0$ and letting $\langle \hdots m_1, m_0 \rangle$ be the $(w,s^q)$-path of $m$,
$m_k \in \fix(w'[s^p])$ where $k$ is the length of $w_0$; i.e., the path has the following form:
\[
m \xleftarrow{\ w_1} m_{k} \xleftarrow{\ w'} m_{k}\stackrel{\ w_0}{\longleftarrow} m
\]
\end{enumerate}
\end{definition}
The reader is invited to check that $\leq_{\Q_{\mathcal G}}$ is transitive---it is for this reason that we require $F^p$ to be closed under taking subwords.

We write any condition  $p \in \Q_{\mathcal G}$ as $(s^p, F^p)$ if we want to refer to the components of that condition.

If $G$ is $(\Vee,\Q_{\mathcal G})$-generic, we let
\[
\sigma_G = \bigcup_{p\in G} s^p.
\]
As we shall see,  $\rho_{\mathcal G,\sigma_G}\colon\mathcal G*\bF(X)\to \langle \mathcal G, \sigma_G \rangle$ is a faithful cofinitary representation; moreover, $\langle \mathcal G, \sigma_G \rangle$ is maximal with respect to the ground model, that is, for no $\tau\in (S_\infty\setminus\mathcal G) \cap \Vee$ is $\mathcal G\cup\{\sigma_G, \tau\}$ contained in a cofinitary group.

\medskip

The role of the requirement regarding fixed points in \eqref{d.Q.simple.order} is to guarantee that 
$\langle \mathcal G, \sigma_G\rangle$ is cofinitary (as will be seen Lemma \ref{l.fp}).

As is pointed out in \cite[p.~42f.]{zhang-maximal}, one cannot  replace  \eqref{d.Q.simple.order} by the simpler
\begin{enumerate}[label={$(\ref{d.Q.simple.order})'$},ref={$(\ref{d.Q.simple.order})'$}]
\item\label{d.Q.simple.order'}
$(s^q,F^q )\leq_{\Q_{\mathcal G}} (s^p,F^p)$ if and only if $s^q \supseteq s^p$, $F^q \supseteq F^p$ and for all $w\in F^p$, we have
$\fix(w[s^q])\subseteq  \fix(w[s^p])$.\label{order.alt}
\end{enumerate}
For with this simpler definition, supposing $g\in\mathcal G$ and $n\in\fix(g)$, the condition $(\emptyset, \{ X^{-1}gX\}) \in\Q_{\mathcal G}$ has no extension $q\in\Q_{\mathcal G}$ with $n\in \ran (s^q )$. 
It is fixed points of conjugate subwords that give rise to this problem. Accordingly, if desired one could replace the phrase `non-empty subword' in \eqref{d.Q.simple.order} by the phrase `\emph{conjugate} subword'.

In a previous article \cite{fischer-toernquist-template} by two of the present authors, allowing only certain words in $F^p$ made it possible to define $\leq_{\Q_{\mathcal G}}$ as in \ref{d.Q.simple.order'}. 
The present definition helps to simplify proofs in comparison with both \cite{fischer-toernquist-template} and Zhang's original version \cite{zhang-maximal}.

\medskip

We now prove two versions of the Domain Extension Lemma and a crucial lemma concerning the length of certain paths (Lemma \ref{l.path.control.simple}).
This will considerably clean up the presentation when we deal with the more complicated forcing 
$\Q^z_{\mathcal G}$.

\medskip

The following is implicit in \cite{zhang-maximal}; for the convenience of the reader, we include a new, very short proof.
\begin{lemma}[Contingent Domain Extension for $\Q_{\mathcal G}$]\label{l.domain.ext.cont}
Suppose $s\in \parfin(\N,\N)$ is injective, $w\in W_{\mathcal G, X}$ and $n\in\N$ is such that for any non-empty subword $w'$ of $w$, $n\notin \fix(w'[s])$. 
Then for a cofinite set of $n'$, letting $s' = s \cup \{ (n,n')\}$, we have that $s'$ is injective and $\fix(w[s'])=\fix(w[s])$.
\end{lemma}
\begin{proof}
Let $W^*$ be the set of subwords of circular shifts of $w$ and pick $n'$ arbitrary such that
\begin{equation}
\begin{split}\label{e.domain.ext.cont}
n' \notin  &\bigcup \big\{ \fix(w'[s]) \colon w' \in W^*\setminus\{ \emptyset\}\big\},\\
n' \notin &\bigcup  \big\{ w'[s]^i (n) \colon i \in \{-1,1\}, w' \in W^* \big\},\text{ and}\\
n' \notin &\ran(s).
\end{split}
\end{equation}
The role of the last requirement is to ensure that $s'$ is injective. 

Assume towards a contradiction that $m_0 \in \fix (w[s']) \setminus \fix(w[s])$.
As the $(w,s)$-path of $m_0$ differs from the $(w,s')$-path, the latter must contain an application of $X$ to $n$ or of $X^{-1}$ to $n'$. Write this latter path as
\begin{equation}\label{e.path}
 \hdots
m_{k(3)}\stackrel{\;\;w''}{\longleftarrow} m_{k(2)} 
\stackrel{\;\;X^{j}}{\longleftarrow} m_{k(1)} 
\stackrel{\;\;w'}{\longleftarrow} m_{k(0)}=m_0
\end{equation}
where $j\in\{-1,1\}$ and $m_{k(1)}=n$ when $j=1$, $m_{k(1)}=n'$ when $j=-1$; moreover we ask that $w',w''\in W_{\mathcal G, X}$ are the maximal subwords of $w$ such that 
from $m_{k(0)}$ to $m_{k(1)}$ and $m_{k(2)}$ to $m_{k(3)}$, the path contains no application of $X$ to $n$ or of $X^{-1}$ to $n'$ (allowing either of $w'$, $w''$ to be empty). 
Thus, $w'$ and $w''$ correspond to path segments where $s$ and $s'$ agree:
\begin{align*}
w'[s] (m_{k(0)})=w'[s'] (m_{k(0)}) = m_{k(1)},\\
w''[s] (m_{k(2)})=w''[s'] (m_{k(2)}) = m_{k(3)}.
\end{align*}

We show that we can assume $w=w''X^jw'$.
Otherwise, by maximality of $w''$, at step $k(3)$ again $X$ is applied to $n$ or $X^{-1}$ to $n'$. Write the path as 
\begin{equation*}
 \hdots
 \stackrel{\;\;\;X^{j'}}{\longleftarrow} m_{k(3)} 
\stackrel{\;\;w''}{\longleftarrow} m_{k(2)} 
\stackrel{\;\;X^{j}}{\longleftarrow} m_{k(1)} 
\stackrel{\;\;w'}{\longleftarrow} m_{k(0)}= m_0
\end{equation*}
with $j'\in\{-1,1\}$ and observe:
\begin{enumerate}[label=\arabic*.]
\item $m_{k(2)} = m_{k(3)}$; for otherwise, $n' = (w'')^i[s](n)$ for some $i\in\{-1,1\}$, contradicting the choice of $n'$.
\item Thus, $w'' \neq \emptyset$, since on one side of $w''$ we have $X$ and on the other $X^{-1}$ and  $w$ is in reduced form.
\item As $n' \notin \fix(w''[s])$, we have that $m_{k(2)}=m_{k(3)}=n$.
\item So $n \in \fix(w''[s])$, contradicting the hypothesis of the lemma.
\end{enumerate}
Thus, $w=w''X^jw'$ and $m_0 = m_{k(3)}$. 
We infer that $n'= (w' w'')^{-j}[s](n)$, again contradicting the choice of $n'$. 
\end{proof}
\begin{remark}
Note as this is easily overlooked, that \eqref{e.domain.ext.cont} implies $n'\neq n$, since $\emptyset\in W^*$ and by Convention \ref{convention}.
Also note that the requirements in \eqref{e.domain.ext.cont} were chosen to be easy-to-state rather than minimal (as will be the case for \eqref{e.domain.ext.simple}, \eqref{e.domain.ext.path}, \eqref{e.domain.ext.coding} and   \eqref{e.generic.hitting} below).
\end{remark}
This enables us to give a short proof of the analogue of \cite[Lemma~2.2]{zhang-maximal}:
\begin{lemma}[Domain Extension for $\Q_{\mathcal G}$]\label{l.domain.ext.simple}
For any $n\in \N$, the set of $q$ such that $n \in \dom(s^q)$, is dense in $\Q$. 
\end{lemma}
\begin{proof}
Fix $p\in \Q$; we shall find a stronger condition $q\in\Q$ such that $n \in \dom(s^q)$.
Analogously to the previous proof, 
let $F^*$ consist of the empty word together with all subwords of circular shifts of  words in $F^p$, and let $n'$ be arbitrary such that
\begin{equation}
\begin{aligned}
n' &\notin  \bigcup \big\{ \fix(w[s]) \colon w \in F^*\setminus\{  \emptyset\}\big\}, \\
n' &\notin  \bigcup \big\{ w[s]^i (n) \colon i \in \{-1,1\}, w \in F^*\big\}, \text{ and}\\
n' &\notin  \ran(s). \label{e.domain.ext.simple}
\end{aligned}
\end{equation}
Note that \eqref{e.domain.ext.simple} excludes only finitely many values for $n'$.
Define $s'=s \cup \{ (n,n')\}$ and $q = (s', F^p)$.
As in the proof of Lemma \ref{l.domain.ext.cont}, $s'$ is injective.

\medskip

Given $w\in F^p$ and supposing $m_0\in \fix(w[s'])\setminus \fix(w[s])$, the proof of the previous lemma shows that there is a subword $w'$ of $w$
such that $n\in \fix(w'[s])$ and $n$ appears in the $(w,s')$-path of $m_0$.
In other words, 
\[
\set(w,s',m)\cap \fix(w'[s])\neq \emptyset,
\]
whence $q \leq_{\Q_{\mathcal G}} p$.
\end{proof}

A crucial observation  for the following discussion of $\Q^z_{\mathcal G}$ is that when extending the domain of $s^p$ for a given condition $p$, we have fine control over the length of paths that result from this extension.

\begin{lemma}[Lengths of paths]\label{l.path.control.simple}
Fix $w\in W_{\mathcal G, X}$ and $m \in \N$.
Moreover, let $s\in\parfin(\N,\N)$ and $n\in \N\setminus\dom(s)$ be given.

Then for cofinitely many $n'\in \N$, if we let $s' = s\cup \{ (n,n')\}$ the following holds:
\begin{enumerate}[label=\arabic*.,ref=\arabic*]
\item \label{case.two} If the $(w, s)$-path  of $m$ does not terminate with last value $n$ before an occurrence of $X$, 
$\mpath(w,s',m)=\mpath(w,s,m)$.

\item \label{case.one} If $w$ is not conjugate to any of its proper subwords, $w \notin \mathcal G$, and the $(w, s)$-path  of $m$ terminates with last value $n$ before an occurrence of $X$, 
the $(w, s')$-path of $m$ contains exactly one more application of $X$ than does the $(w, s)$-path and no further application of $X^{-1}$.
\end{enumerate}
\end{lemma}
\begin{proof}
Let $E = \dom(s) \cup\ran(s)\cup\{n\}\cup\{m\}$ and $W^*$ be the set of subwords of circular shifts of $w$. 
Suppose $n'$ is arbitrary such that
\begin{equation}
\begin{split}\label{e.domain.ext.path}
n' \notin  &\bigcup \big\{ \fix(w'[s]) \colon w' \in W^*\setminus\{ \emptyset\}\big\} \text{ and}\\
n' \notin &\bigcup  \big\{ w'[s]^i [E] \colon i \in \{-1,1\}, w' \in W^* \big\}.\\
\end{split}
\end{equation}

In Case \ref{case.two} of the lemma, show $\mpath(w,s',m)=\mpath(w,s,m)$.
Suppose that the $(w, s)$-path  of $m$ terminates after $k$ steps with last value $m_k$ before an occurrence of $X^j$, $j\in\{-1,1\}$.
If $j=1$, $m_k\neq n$ by assumption; and as the $(w, s)$-path terminates with $m_k$, we have $m_k\notin\dom(s)\cup\{ n\}$, so the $(w, s')$-path terminates as well.

So assume towards a contradiction that $j=-1$ and $m_{k+1}$ in the $(w, s')$-path  of $m$ is defined.
As  the $(w, s)$-path terminates with $m_k$, before an occurrence of $X^{-1}$, while $(w,s')$-path does not terminate, $m_k=n'$.
Thus, $n' \in w'[s] [E]$ for a subword $w'$ of $w$, contradicting \eqref{e.domain.ext.path}.

\medskip

For Case \ref{case.one} of the lemma, suppose that
the $(w, s)$-path  of $m$ terminates after $k$ steps with last value $m_k=n$ before an occurrence of $X$.
In the $(w,s')$-path we have on the contrary that $m_{k+1}=n'$.

If the letter occurring at step $k+2$ in this path is again $X$,
the path terminates after $k+1$ steps with last value $n' = m_{k+1}$, since $n'\notin\dom(s')$ by \eqref{e.domain.ext.path}.

If the letter occurring at step $k+2$ is $g\in\mathcal G\setminus\{ \one\}$ followed by $X$ or $X^{-1}$ at step $k+3$, the path terminates after $k+2$ steps with last value $m_{k+2}= g(n')$, as otherwise, $n' \in g^{-1}[\dom(s')\cup\ran(s')]$, which implies
$n' \in g^{-1}[E]$
 or $n' \in \fix(g)$,
 contradicting \eqref{e.domain.ext.path}.
 
Lastly, suppose the letter occurring at step $k+2$ is $g\in\mathcal G\setminus\{ \one\}$ and $g$ is the last letter of $w$.
Further suppose  the following letter at step $k+3$ (the first letter in $w$) is $h \in \mathcal G\setminus\{\one,g^{-1}\}$  followed by $X$ or $X^{-1}$ at step $k+4$. 
Then the path terminates after $k+3$ steps with last value $m_{k+3}= h(g(n'))$, as otherwise, $n' \in (h g)^{-1}[\dom(s')\cup\ran(s')]$, which implies
$n' \in (hg)^{-1}[E]$
 or $n' \in \fix(hg)$,
 contradicting \eqref{e.domain.ext.path}.
 
As $w$ has no proper conjugate subwords and is reduced no other combination of letters can occur at the next steps.
\end{proof}

All other proofs regarding $\Q_{\mathcal G}$ will be omitted (but note that they can be inferred from their counterparts for $\Q^z_{\mathcal G}$ in the next section).

\subsection{Coding into a generic cofinitary group extension}\label{ss.Q}

Our next goal is to define, given $z \in 2^\N$, a forcing $\Q^z_{\mathcal G}$ such that whenever $G$ is $(\Vee, \Q^z_{\mathcal G})$-generic the following holds: There exists $\sigma_{G}\in S_\infty$ such that for a large enough set of 
$\sigma \in \langle \mathcal  G,\sigma_{G} \rangle$ (i.e., large enough for our application), $z$ is computable from $\sigma$ (i.e.\ the characteristic function of $z$ is computable by a Turing machine using $\sigma$ as an oracle).

\medskip

First, we describe the algorithm by which $z$ is computed from an element of $\langle \mathcal  G,\sigma_{G} \rangle\setminus\mathcal G$.
Since our forcing uses finite approximations to $\sigma_G$, we must also define the coding for elements of $\parfin(\N,\N)$.

\begin{definition}[Coding]\label{d.coding}~
\begin{enumerate}

\item
We say that $\sigma\colon\N\to\N$ \emph{codes} $z \in 2^\N$ \emph{with parameter} $m$ if and only if
\[
(\forall k\in \N) \; \sigma^{k+1}(m) \equiv z(k) \pmod{2}.
\] 

\item Let $t \in \{0,1\}^l$, where $l\in \N$.
We say that $\sigma \in \parfin(\N,\N)$  \emph{exactly codes} $t$
 \emph{with parameter} $m\in \N$ if and only if
\[
(\forall k< l) \; \sigma^{k+1}(m) \equiv t(k) \pmod{2}
\]
and in addition $\sigma^{l+1}(m)$ is undefined.
\end{enumerate}
\end{definition}
Note that using this algorithm when $\sigma \in S_\infty$ is 
conjugate to an element of $\mathcal G$ with only finite orbits, we can't expect that arbitrary $z\in 2^\N$ be coded---the sequence coded by such $\sigma$ will always be periodic.
Fortunately we can exclude such $\sigma$ from consideration for the proof of our main result.

In fact it suffices for our application to restrict our attention to those $\sigma$ which can be represented as a word in $W_{\mathcal G, \sigma_G}\setminus\mathcal G$ which is not conjugate to any of its proper subwords (we will see below that $\rho_{\mathcal G, \sigma_G}$ is injective so in fact this representation is unique).\footnote{Devising a forcing which similarly codes $z$ into all words except those conjugate to elements of $\mathcal G$ with only finite orbits is possible, but more involved for several reasons.}

\medskip

For the rest of this section, fix $z\in 2^\N$.
Now we can define the forcing. 
 \begin{definition}[Definition of $\Q=\Q^z_{\mathcal G}$]\label{d.Q}~
\begin{enumerate}[label=(\Alph*),ref=\Alph*]
\item\label{conditions.Q}
Conditions of $\Q$ are triples $p=(s^p,F^p, \bar m^p)$ s.t.\
\begin{enumerate}[label=(\arabic*),ref=\ref{conditions.Q}\arabic*]
 \item $(s^p,F^p) \in \Q_{\mathcal G}$

\item $\bar m^p $ is a finite partial function into $\N$ from the set of words $w \in W_{\mathcal G,X}\setminus\mathcal G$ such that no proper subword $w'$ of $w$ is conjugate to $w$.

\item\label{code} For any $w\in \dom(\bar m^p)$ there is $l\in\omega$ such that $w[s^p]$ exactly codes $z\restriction l$ with parameter $\bar m^p(w)$.

\item\label{disjoint} If $w, w' \in \dom(\bar m^p)$ and $w \neq w'$, 
\[
\set(w, s^p, \bar m^p(w)) \cap \set(w', s^p, \bar m^p(w')) = \emptyset
\]
\end{enumerate}
\item\label{order.Q}
$(s^q,F^q,  \bar m^q)\leq (s^p,F^p,  \bar m^p)$  if and only if
\begin{enumerate}[label=(\arabic*),ref=\ref{order.Q}\arabic*]
\item $(s^q,F^q) \leq_{\Q_{\mathcal G}} (s^p,F^p)$,
\item $\bar m^q$ extends $\bar m^p$ as a function.
\end{enumerate}
\end{enumerate}
\end{definition}
\noindent
We write any condition  $p \in \Q$ as $(s^p, F^p, \bar m^p)$ if we want to refer to the components of that condition.

Note that \eqref{code} ensures that for any $p\in \Q$ and $w \in \dom(\bar m^p)$ the path under $(w,s^p)$ of $\bar m^p(w)$ is finite (although other paths may be eventually periodic and thus infinite).
Also note that by \eqref{disjoint}, $\vert\mathcal G\vert$ is collapsed to $\omega$ by $\Q$ whenever $\mathcal G$ is uncountable in the ground model.

\medskip

For a $(\Vee,\Q)$-generic $G$, as in the previous section we let
\[
\sigma_G = \bigcup_{p\in G} s^p.
\]
We now show in a series of lemmas that $\rho_{\mathcal G,\sigma_G}\colon\mathcal G*\bF(X)\to \langle \mathcal G, \sigma_G \rangle$ is a faithful cofinitary representation and for any $w \in W_{\mathcal G, X}\setminus\mathcal G$ which does not have a proper conjugate subword, 
$w[\sigma_G]$ codes $z$.
The reader should note that this immediately implies that for any $g \in \mathcal G\setminus \{\one\}$ and any $\sigma\in \langle \mathcal G, \sigma_G \rangle$ either $z$ is coded by $\sigma$ or by $g\sigma$.

\medskip

We begin with a Lemma showing that $\sigma_{G}$ is forced by $\Q$ to be totally defined on $\N$.

\begin{lemma}[Domain Extension]\label{l.domain.ext.coding}~
\begin{enumerate}[label={\arabic*.},ref=\arabic*]
\item For any $n\in \N$, the set $D^{\textup{d}}_n = \{q\in\Q \;\colon n \in \dom(s^q)\}$ is dense in $\Q$.

\item\label{l.strong.version} In fact, suppose $p\in \Q$, $n\in \N$, $w^*\in W_{\mathcal G, X}\setminus \mathcal G$ is not conjugate to any of its proper subwords, 
$n$ is the last value of $\mpath(w^*,s^p, m^*)$ and this path terminates before an occurrence of $X$.
Then one can find $q\in \Q$ such that $q\leq p$ and $\mpath(w^*,s^q,  m^*)$ contains exactly one more application of $X$, and no further application of $X^{-1}$, than does $\mpath(w^*,s^p, m^*)$.
\end{enumerate}
\end{lemma}

Before we prove the lemma, to avoid repetition, we introduce the following terminology:
For $w \in W_{\mathcal G,X}$ and $j\in \{-1,1\}$,
call an occurrence of $X^j$ in $w$ \emph{critical} if there is no occurrence of $X$ or $X^{-1}$ in $w$ to its left.
Otherwise, we call it an \emph{uncritical occurrence}.
Clearly, it is through a critical occurrence of $X$ (resp.\ $X^{-1}$) in some word in $\dom(\bar m)$ that the coding requirements from \eqref{code} restrict our possibilities to extend $\dom(s^p)$ (resp.\ $\ran(s^p)$).

\begin{proof}[Proof of Lemma \ref{l.domain.ext.coding}]
Let $p\in \Q$, $w^*\in W_{\mathcal G, X}$, and $m^*, n\in \N$ be as in the statement of the lemma and suppose $n \not\in \dom(s^p)$.
We will find $n'$ such that for $s' = s^p \cup \{ (n,n')\}$, $q=( s', F^p,  \bar m^p)$ is a condition stronger than $p$.

\medskip

Write $s$ for $s^p$ and $\bar m$ for $\bar m^p$.
Let 
\[
E= \ran(s) \cup\dom(s) \cup\{ n\} \cup \ran(\bar m)\cup \{ m^*\},
\]
and let $F^*$ consist of all words which are a subword of a circular shift of a word in $F^p\cup\{ w^*\}$. 
The first requirement we make is that $n'$ be chosen such that
\begin{equation}\label{e.domain.ext.coding}
\begin{aligned}
n' &\notin  \bigcup \{ \fix(w[s]) \colon w \in F^*\setminus\{ \emptyset\}\} \text{ and}\\
n' &\notin  \bigcup \big\{ g^{-1}w^i[s]  [E] \colon i \in \{-1,1\}, w \in F^*, g\in F^*\cap\mathcal G \big\}.
\end{aligned}
\end{equation}
Note that \eqref{e.domain.ext.coding} excludes only finitely many possible values for $n'$.
The taking of preimages under $g\in F^*\cap\mathcal G$ in \eqref{e.domain.ext.coding} 
serves the sole purpose of ensuring that the following holds:
\begin{equation}\label{e.n'.unused}
g(n') \notin  \bigcup \{  \set(w,s, \bar m(w)) \colon w\in \dom(\bar m) \}
\end{equation}
for $g=\one$ as well as for all $g\in\mathcal G\setminus\{ \one\}$ occurring in a word from $F^p$.
\medskip

Now suppose for some $w \in \dom(\bar m)$, $n$ appears in the $(w,s)$-path of $\bar m(w)$ before a critical occurrence of $X$. 
In this case, we must make an additional requirement:
Noting that by \eqref{disjoint} there is at most one such $w\in \dom(\bar m)$, let this unique word be denoted by $w^{**}$.
Let $l$ be such that $w^{**}[s]$ exactly codes $z\restriction l$ with parameter $\bar m(w^{**})$.
Further, suppose $w^{**}=gXw'$, where $w' \in W_{\mathcal G, X}$ and we allow $g \in \mathcal G$ to be $\one$ but no cancellation in $ X w'$.
Now in addition to \eqref{e.domain.ext.coding}, require that $g(n')$ be even if $z(l)=0$ and odd if $z(l)=1$. 
This is possible as \eqref{e.domain.ext.coding} excludes only finitely many values.
Note that by choice of $n'$ and Lemma~\ref{l.path.control.simple}, only for this one $w \in \dom(\bar m^q)$ does the $(w,s')$-path differ from the $(w,s)$-path; and as $w$ has no proper conjugate subword, the $(w,s')$-path does not go past the next critical occurrence of $X$ or $X^{-1}$. Thus \eqref{code} holds of $q$ by construction.

\medskip

To see that $q$ is a condition, it remains to verify  \eqref{disjoint}. 
We have seen in the proof of Lemma~\ref{l.path.control.simple} that 
 at most one path starting in $\dom(\bar m^p)$ acquires new values when passing to $q$, and these new values are $n'$ and possibly  $g(n')$,
where $g\in \mathcal G\setminus\{\one\}$ is a subword of a cyclic shift of a word in $\dom(\bar m)$.
Thus requirement \eqref{disjoint} holds of $q$ by \eqref{e.n'.unused}.

We end the proof of Part 1 of the lemma by quoting the proof of the Domain Extension Lemma for $\Q_{\mathcal G}$ to conclude that  $q \leq p$.

\medskip

For Part \ref{l.strong.version} of the lemma note that by the proof of Lemma \ref{l.path.control.simple}
indeed the $(w^*,s)$-path of $m^*$ contains exactly one more application of $X$, and no further application of $X^{-1}$, than does the $(w^*,s')$-path of $m^*$.
\end{proof}

The next lemma shows that $\Q$ forces $\sigma_G$ to be onto $\N$.
\begin{lemma}[Range Extension]\label{l.range.ext}~
\begin{enumerate}[label=\arabic*.]
\item For any $n\in \N$, the set $D^{\textup{r}}_n = \{q\in\Q \;\colon n \in \ran(s^q)\}$ is dense in $\Q$. 

\item In fact, suppose $p\in \Q$, $n\in \N$, and $m^*\in \N$ are such that for some $w^*\in W_{\mathcal G, X}\setminus\mathcal G$ not conjugate to any of its proper subwords,
$n$ is the last value of $\mpath(w^*,s^p, m^*)$ and this path terminates before an occurrence of $X^{-1}$.
Then one can find $q\in \Q$ such that $q\leq p$ and $\mpath(w^*,s^q, \bar m^*)$ contains exactly one more application of $X^{-1}$, and no further application of $X$, than does $\mpath(w^*,s^p, \bar m^*)$. 
\end{enumerate}
\end{lemma}
\begin{proof}
The lemma is entirely symmetrical to the Domain Extension Lemma.
By symmetry, the proofs of Lemmas \ref{l.domain.ext.cont}, \ref{l.domain.ext.simple}, \ref{l.path.control.simple} and \ref{l.domain.ext.coding} can easily be adapted.
\end{proof}

By the previous two lemmas, $\forces_{\Q} \sigma_{\dot G} \in S_\infty$.
By the next lemma it is forced that for all $w \in W_{\mathcal G, X}\setminus \mathcal G$ which are not conjugate to any of their proper subwords, $w[\sigma_{\dot G}]$ codes $z$,
as promised.
\begin{lemma}[Generic Coding]\label{l.coding}
For any $w \in W_{\mathcal G,X}\setminus \mathcal G$ without a proper conjugate subword and any $l \in \N$, let $D^{\textup{code}}_{w,l}$ denote the set of $q\in\Q$ such that $w \in \dom(\bar m^q)$ and for some $l' \geq l$,
$q$ exactly codes $z\restriction l'$ with parameter $\bar m^q(w)$. Then $D^{\textup{code}}_{w,l}$ is dense in $\Q$.
\end{lemma}
\begin{proof}
Let $w$ and $l$ be as above, and fix $p\in \Q$.
We may assume $w \in \dom(\bar m^p)$: Otherwise, find $n'\in\N$ such that $n' \equiv z(0) \pmod{2}$ and   \eqref{e.domain.ext.coding} holds with $E=\dom(s^p)\cup\ran(s^p)\cup\ran(\bar m^p)$ and $F^*$ equal to the set of subwords of circular shifts of words in $F^p\cup\{w\}$, and let $p' = (s^p, F^p, \bar m^p\cup\{ (w, n')\})$.
By \eqref{e.domain.ext.coding} and by the proof of Lemma \ref{l.domain.ext.coding}, \eqref{disjoint} is satisfied for $p'$ and the $(w,s^{p})$-path of $n'$ will terminate before the right-most application of $X$ or $X^{-1}$ in $w$.
As $s^p = s^{p'}$, this suffices to show $p'$ is a condition stronger than $p$.

So supposing $w\in \dom(\bar m^p)$, let $m$ be the last value of the $(w,s^p)$-path of $\bar m(w)$ and assume this path terminates before an occurrence of the letter $X$.
By the Domain Extension Lemma, we may find $q \leq p$ such that $m \in \dom(s^{q})$ and the $(w,s^{q})$-path at
$\bar m(w)$ terminates either at the next step or after one further application of an element of $\mathcal G\setminus\{\one\}$.

If instead the $(w,s^p)$-path of $\bar m(w)$ terminates before an occurrence of the letter $X^{-1}$,
argue similarly using the Range Extension Lemma.

Repeating the argument if necessary, we obtain a condition $q$ such that $s^q$ exactly codes $z\restriction l$.
\end{proof}

By the next two lemmas 
$\langle \mathcal G, \sigma_G\rangle$ is forced to be cofinitary.
The reader may care to notice that the proofs of the remaining lemmas in this section are almost identical (or in the case of Lemma~\ref{l.rho.injective} at least not very different) for  $\langle \Q_{\mathcal G}, \leq_{\mathcal G}\rangle$---excepting of course Lemma~\ref{t.Q}\eqref{t.Q.coding} which doesn't hold for  $\langle \Q_{\mathcal G}, \leq_{\mathcal G}\rangle$.
\begin{lemma}\label{l.words}
If $w\in W_{\mathcal G,X}$, the set $\{q\in\Q \;\colon w \in F^q)\}$ is dense in $\Q$. 
\end{lemma}
\begin{proof}
Given  $p\in\Q^z_{\mathcal G}$, let $q=(s^p, F^p\cup W,\bar m^p)$, where $W$ is the set of subwords of $w$. Then
$q \in\Q^z_{\mathcal G}$ and $q \leq p$.
\end{proof}
\begin{lemma}\label{l.fp}
If $p\in \Q$ and $w \in F^p$,
there is $N$ such that
$p\forces \fix(w[\sigma_{\dot G}])$ has size at most $N$.
\end{lemma}
\begin{proof}
Suppose $p$ and $w\in F^p$, and let the length of $w$ be $k$.
Let $N$ be the number of triples $(u,m,l)$ such that  $u$ is a subword of $w$, $m \in \fix(u[s^p])$, and $l \leq k$.
Towards a contradiction, assume $q \leq p$ forces that
$(w[\sigma_{\dot G}])$ has more than $N$ fixed points.
By strengthening $q$ if necessary, we may assume that $\left\vert \fix(w[s^q])\right\vert > N$.

As $q\leq p$, by \eqref{d.Q.simple.order} in Definition \ref{d.Q.simple}, for each $n\in \fix(w[s^q])$ 
there is a non-empty subword $w'$ of $w$ satisfying 
\[
\set(w,s^q,n) \cap \fix(w'[s^p])\neq \emptyset.
\]
So letting $\langle \hdots, m_1, m_0 \rangle$ be the $(w,s^q)$-path of $n$, pick a subword $u=u(n)$ of $w$, $m=m(n)\in \fix(u[s^p])$ and $l=l(n)\leq k$ so that
$m_l = m$, for each such $n$.

Note that the function $n \mapsto ( u(n),m(n), l(n) )$ is injective (since $w[s^q]$ is);
but it maps $\fix(w[s^q])$ to a set of size $N$, contradiction.
\end{proof}

The next lemma shows that our construction yields a group which is 
maximal with respect to permutations from the ground model:
For any $\tau \in S_\infty\setminus \mathcal G$,
$\forces_{\Q^z_{\mathcal G}}\langle\mathcal G,\sigma_{\dot G},\rangle\cup\{\check\tau\}$ is not contained in a cofinitary group.
In fact, the Lemma shows: For any $\tau \in S_\infty$ such that $\langle \mathcal G, \tau\rangle$ is cofinitary,
$\forces_{\Q^z_{\mathcal G}}\text{``}\{ n \colon \sigma_{\dot G}(n)= \check\tau(n)\}$ is infinite\text{''}.
We will prove a generalization, the $\P$-generic Hitting Lemma below.
\begin{lemma}[Generic Hitting]\label{l.generic.hitting}
Given $m\in\N$ and $\tau \in S_\infty$ such that $\langle \mathcal G, \tau\rangle$ is cofinitary, the set $D^{\textup{hit}}_{\tau,m} = \{q\in \Q\;\colon (\exists n\geq m) \; s^q(n)=\tau(n)\}$ is dense.
\end{lemma}
\begin{proof}
Let $p\in \Q$ and $m\in \N$ be given.
Find $n \in \N$ such that $n\geq m$, and such that letting $n'=\tau(n)$, $n'$ satisfies \eqref{e.domain.ext.coding} with $s=s^p$, $E=\dom(s^p)\cup\ran(s^p)\cup\{ n\}\cup\ran(\bar m)$ and $F^*$ equal to the set of subwords of circular shifts of words in $F^p$.

To see this is possible, note that \eqref{e.domain.ext.coding} holds for $n'=\tau(n)$ if and only if for $E'=\dom(s)\cup\ran(s)\cup\ran(\bar m)$,
\begin{equation}\label{e.generic.hitting}
\begin{split}
&n \notin  \tau^{-1}\Big[ \bigcup \big\{ \fix(w[s]) \colon w \in F^*\setminus\{ \emptyset\}\big\}\Big],\\
&n \notin  \tau^{-1} \Big[ \bigcup \big\{ g^{-1}w'[s]^i [E'] \colon i \in \{-1,1\}, w' \in F^*, g\in F^*\cap\mathcal G \big\}\Big],\text{ and}\\
&n \notin \bigcup \big\{ \fix(\tau^{-1} g^{-1}w'[s]^i) \colon i \in \{-1,1\}, w' \in F^*, g\in F^*\cap\mathcal G \big\}.
\end{split}
\end{equation}
The first two requirements obviously exclude only finitely many $n$;
the same holds for the last requirement in case $w' \notin\mathcal G$, when $w'[s]$ is finite.
Lastly, when $w' \in\mathcal G$, the last requirement excludes only finitely many $n$ as $\langle \mathcal G, \tau\rangle$ is cofinitary.
Thus, $n$ as above can indeed be found.

By the proof of the Domain Extension Lemma, letting $s' = s^p\cup\{ (n,\tau(n))\}$, $q=(s', F^p)$ is a condition stronger that $p$. 
\end{proof}
\begin{remark}\label{r.indestructible}
By standard properties of product forcing, the previous lemma is easily seen to imply the following:
If $G$ is $(\Vee, \Q^z_{\mathcal G})$-generic and $H$ is $(\Vee[G],\P)$-generic for a forcing $\P\in\Vee$ then $\langle \mathcal G, \sigma_G\rangle$ is maximal with respect to $S_\infty\cap \Vee[H]$,
i.e.\ if $\tau \in (S_\infty\setminus \mathcal G)\cap \Vee[H]$ there is no cofinitary group $\mathcal G' \in \Vee[G][H]$ such that $\mathcal G\cup\{ \sigma_G, \tau\} \subseteq \mathcal G'$.
This observation inspired our construction of a Cohen-indestructible mcg in the next section.
\end{remark}

The following immediately implies that
$\Q$ forces that $\rho_{\mathcal G, \sigma_G}$ is injective.
\begin{lemma}\label{l.rho.injective}
For any $w \in W_{\mathcal G,X}\setminus\{\one\}$, the set $D^{\textup{1-1}}_w$ of $q$ such that
$q \forces_{\Q} w[\sigma_{\dot G}] \neq \one$ is dense.
\end{lemma}
\begin{proof}
Fix $p \in \Q$ and $w \in W_{\mathcal G,X}\setminus\{\one\}$.
We may assume $w \notin \mathcal G$ and $w$ has no proper conjugate subwords.
Find $l$ such that $w^l \notin \dom(\bar m^p)$ and $q \in \Q$ stronger than $p$ such that
for some $m \not \equiv z(0) \pmod 2$ we have $m = \bar m^q(w^l)$ (this is possible as $m$ can be chosen from a cofinite set).
Then $q \forces_{\Q} w[\sigma_{\dot G}] \neq \one$ ($\one$ codes 
the sequence with constant value $m \bmod 2$ with parameter $m$).
\end{proof}

\medskip

Anticipating that in the next section we apply our forcing $\Q^z_{\mathcal G}$ over countable initial segments of $\eL$ satisfying only a very weak fragment of ZFC, we list crucial properties of $\Q^z_{\mathcal G}$ when forcing over such models:
\begin{lemma}\label{t.Q}
Suppose
$M$ is a transitive $\in$-model such that $\Q^z_{\mathcal G}\in M$, and for any $m \in \omega$, $\tau\in S_\infty\cap M$ such that $\langle \mathcal G,\tau\rangle$ is cofinitary, and $w \in W_{\mathcal G,X}\setminus \mathcal G$ we have 
$\{ D^{\textup{d}}_m, D^{\textup{r}}_m, D^{\textup{hit}}_{\tau,m}, D^{\textup{1-1}}_w \} \subseteq M$ and $D^{\textup{code}}_{w,m} \in M$ when $w$ has no proper conjugate subwords.
Then for any $(M,\Q^z_{\mathcal G})$-generic filter $G$, 
letting
\[
\sigma_G = \bigcup_{p\in G} s^p
\]
the following holds:
\begin{enumerate}[label=(\Roman*),ref=\Roman*]
\item\label{t.Q.cof} $\rho_{\mathcal G, \sigma_G}\colon \mathcal G * \bF(X) \to \langle \mathcal G, \sigma_G\rangle$ is a faithful cofinitary representation.
\item\label{t.Q.coding} For any word $w \in W_{\mathcal G, X}\setminus\mathcal G$ which does not have a proper subword $w'$ which is conjugate to $w$, we have that $w[\sigma_G]$ codes $z$ in the sense of definition \ref{d.coding}.
\item\label{t.Q.max} For any $\tau \in \cofin(S_\infty) \cap M$ such that $\tau\notin\mathcal G$, there is no cofinitary group $\mathcal G'$ such that
$\langle \mathcal G, \sigma_G\rangle \cup \{ \tau \} \subseteq \mathcal G'$.
\end{enumerate}
\end{lemma}
\begin{proof}
As $G$ intersects $D^{\textup{d}}_m$ for each $m\in\N$, $\sigma_G$ is total.
By analogous arguments, \eqref{t.Q.cof}, \eqref{t.Q.coding} and \eqref{t.Q.max}
are obtained using $D^{\textup{r}}_m$, $D^{\textup{code}}_{w,m}$, $D^{\textup{hit}}_{\tau,m}$, and $D^{\textup{1-1}}_w$.
\end{proof}

The next lemma, like its precursor \cite[Theorem 4.1]{fischer-maximal}, will help us to show Cohen-indestructibility in \S\ref{s.main}, in a situation where as above, the ground model satisfies only a weak fragment of ZFC (see also Remark \ref{r.indestructible}).

\begin{lemma}[$\P$-generic hitting]\label{l.cohen.generic.hitting}
Let an arbitrary forcing $\P$, a $\P$-name $\dot \tau$, a condition $(p,q) \in \P\times\Q$ and $k\in \omega$ be given and
suppose $p\forces_{\P}\text{``} \dot \tau\in S_\infty$ and $\langle \check{\mathcal G}, \dot \tau\rangle$ is cofinitary\text{''}.

Then there is $(p',q')\in\P\times\Q$ such that
$(p',q')\leq_{\P\times\Q} (p,q)$ and
\[
(p',q')\forces_{\P\times\Q} (\exists n\in\N)\; {n> k}\wedge{\sigma_{\dot G}(n)=\dot\tau(n)}.
\]
\end{lemma}
\begin{proof}
Fix $\P$, $\dot\tau$, $(p,q)\in \P\times\Q$ and $k$ as in the statement of the lemma.
Let $G$ be $(\P,\Vee)$-generic and such that $p\in G$ and let $\tau = \dot \tau^G$.

Working in $\Vee[G]$,  follow the proof of the Generic Hitting Lemma: 
Find $n$ such that for $s=s^q$, $F^*$ equal to the set of subwords of circular shifts of a word in $F^q$ and  $E'=\dom(s^q)\cup\ran(s^q)\cup\ran(\bar m)$,
\eqref{e.generic.hitting} holds; in addition, require $\tau (n) > k$.
Letting $n'=\tau(n)$, find $p' \in G$ extending $p$ such that
\[
p'\forces_{\P} \dot \tau(\check n) = \check{n'}.
\]
Just as in the proof of the Generic Hitting Lemma, by choice of $n$, we have that for $E=E'\cup\{n\}$, \eqref{e.domain.ext.coding} holds in $\Vee$.
By the proof of the Domain Extension Lemma we can extend $q$ to $q'\in \Q$ such that
\[
q' \forces_{\Q} \sigma_{\dot G}(\check n)=\check{n'},
\]
and we are done.
\end{proof}

\section{A co-analytic Cohen-indestructible  mcg}\label{s.main}

We now use the ideas from the previous section to prove the main result of this paper.
At the same time, we give a new proof of Kastermans' result that there is a $\Pi^1_1$ mcg in $\eL$, based on the idea of finding generics over countable models.

\begin{theorem}\label{t.main}
Assume $\Vee=\eL$.
Let $\mathcal G_0$ be any countable cofinitary group, and fix $c\in 2^\N$ such that an enumeration of $\mathcal G_0$ is arithmetical in $c$ as a subset of $\N\times\N^\N$.
Then
there is a Cohen-indestructible $\Pi^1_1( c)$ maximal cofinitary group which contains $\mathcal G_0$ as a subgroup.
\end{theorem}
Note that for appropriately chosen $\mathcal G_0$, our method will produce a group which is isomorphic to Kastermans' group from \cite{kastermans-complexity}. 
On the other hand, they are not identical as subgroups of $S_\infty$, as one should not expect that his construction produce a Cohen-indestructible group.

\medskip 

Our argument is of the same type as Miller's classical construction given in \cite[\S 7]{miller-infinite}.
A very detailed exposition of this technique can be found in \cite[\S 3]{fischer-co-analytic}; the present account is parallel where possible.
We start by fixing notation and reviewing some facts,  in the course of which we also give a sketch of the proof of Theorem \ref{t.main}.

\medskip

The canonical well-ordering of $\eL$ is denoted by $\leq_{\eL}$.
A \emph{formula} is always a formula in the language of set theory (\cite[p.~1]{kanamori}); likewise for \emph{sentences}. 

\medskip

Given $x\in 2^\N$, let $E_x \subseteq \omega^2$ be the binary relation defined by
\[
m \mathbin{E_x} n \iff x(2^m 3^n)=0. 
\]
If it is the case that $E_x$ is well-founded and extensional, we denote by $M_x$ the set and by $\pi_x$ the map such that $\pi_x\colon \langle \omega, E_x\rangle \to \langle M_x, \in\rangle$ is the unique isomorphism of  $\langle \omega, E_x\rangle $ with a transitive $\in$-model. Recall:

\begin{fact}[see {\cite[13.8]{kanamori}}]\label{f.truth} If $E_x$ is well-founded and extensional
and $\phi$ is a formula  with $k$ free variables,
the following relations are arithmetical in $x$:
\begin{gather*}
\{ \left( m_1,\hdots,m_k \right) \in  \N \times\hdots \times \N \colon \langle M_x, \in\rangle\vDash \phi(m_1,\hdots,m_k) \},\\
\{ \left(  y, m\right) \in \N^\N \times \N \colon \pi_x(m)=y \}.
\end{gather*}
\end{fact}

\medskip

Recall that for $x,y\in 2^\N$, $y \in \Delta^1_1(x)$ means that $y$ is a Borel subset of $\N$ with a code recursive in $x$ (i.e.\ $y$ is hyperarithmetic in $x$; see e.g.\ \cite{mansfield-recursive}).
We shall use the following facts:
\begin{fact}\label{f.next}
  If $M_z = \eL_\delta$, there is $y \in \Delta^1_1(z)$ such that
$M_y = \eL_{\delta+\omega+\omega}$ (an elementary proof is implicitly given in \cite[3.6]{kastermans-analytic}).
\end{fact}

\begin{fact}[Mansfield-Solovay, see {\cite[Corollary~4.19,~p.~53]{mansfield-recursive}}]\label{f.mansfield-solovay}
 For any $\Pi^1_1(z)$ formula $\Psi(x)$, the formula $(\exists x \in \Delta^1_1(z)) \; \Psi(x)$ is equivalent to a $\Pi^1_1(z)$ formula .
\end{fact}

 We can now give a brief sketch of our construction of a maximal cofinitary group which is co-analytic and Co\-hen-\-in\-de\-struct\-ible:

If $\Vee=\eL$, one may use the definable well-ordering to enlarge a given countable cofinitary group $\mathcal G_0$ to a mcg:
Assuming by induction we have constructed $\mathcal G_\xi$, where $\xi <\omega_1$, let $\sigma_\xi$ be the $\leq_{\eL}$-least $\sigma\in S_\infty\setminus\mathcal G_\xi$ such that $\mathcal G_{\xi+1} = \langle \mathcal G_\xi, \sigma \rangle$ is cofinitary. Then $\mathcal G= \bigcup_{\xi<\omega_1} \mathcal G_{\xi}$ will be maximal cofinitary.
Using the above correspondence of countable models with elements of $2^\N$, one finds $\mathcal G$ to be $\Sigma^1_2$ in a code for $\mathcal G_0$.

We can `bound the existential quantifier' by altering the above construction so that at step $\xi$, an initial segment of $\eL$ witnessing the fact that $\sigma_\xi \in \mathcal G$ can be found  \emph{effectively}  in some data $z$ (i.e.\ is hyperarithmetic in $z$---here we use Fact \ref{f.next}); at the same time applying ideas from the previous section to ensure that $z$, in turn, can be found effectively in $\sigma_\xi$.
The resulting mcg is shown to be $\mathbf{\Pi}^1_1$ using Fact \ref{f.mansfield-solovay}.

\medskip

The `data' alluded to in the above sketch will be some $z\in2^\N$ such that $M_z$ `sees' the first $\xi$ steps of the construction.
The question of how to uniquely pick $z$ leads naturally to the notion of \emph{projectum}:
We say \emph{$\delta$ projects to $\omega$} if and only if there is a surjection from $\omega$ onto $\eL_\delta$ which is $\mathbf{\Sigma}_1$-definable (with a parameter) in $\eL_\delta$.

\begin{fact}[see \cite{jensen-fine}]\label{f.project}  The set of $\delta$ which project to $\omega$ is cofinal in $\omega_1$.
Moreover, if $\delta$ projects to $\omega$, so does $\delta+\omega$.
\end{fact}
Supposing $\delta$ projects to $\omega$, let $p$ be $\leq_{\eL}$-least parameter such that there is a surjection as above which is $\Sigma_1$  in $p$. Call the surjection $f$ so obtained the \emph{canonical surjection from $\omega$ onto $\eL_\delta$}.

\medskip

Finally, we proceed to prove our main theorem.
\begin{proof}[Proof of Theorem \ref{t.main}]
Assume $\Vee=\eL$, and
fix $\mathcal G_0$ and $c$ as in the theorem.
Since the argument relativizes to the parameter $c$, 
we may  suppress it and assume that an enumeration of $\mathcal G_0$ is arithmetic.
We may also suppose that $\mathcal G_0 \neq \{\one\}$---otherwise, just replace it by a larger arithmetic cofinitary group.

\medskip

We construct a sequence $\vec{\mathcal G}=\langle (\delta_\xi, z_\xi, \mathcal G_\xi, \sigma_\xi) \colon \xi < \omega_1 \rangle$ such that
\begin{itemize}
\item\label{clause.i} ${\delta_\xi}$ is a countable ordinal,
\item\label{clause.ii} $z_\xi \in 2^\N\cap \eL_{{\delta_\xi}+\omega}$,
\item\label{clause.G} $\mathcal G_\xi$ is a countable cofinitary group, and
\item\label{clause.iii} $\sigma_\xi \in S_\infty$,
\end{itemize}
and so that the following hold for each $\xi <\omega_1$:
\begin{enumerate}[label=(\roman*),ref=\roman*]

\item\label{clause.delta} ${\delta_\xi}$ is the least ordinal $\delta > \sup_{\nu<\xi} \delta_\nu$ such that
$\delta$ is a limit of limit ordinals and projects to $\omega$.

\item\label{clause.z} $z_\xi$ is obtained from the canonical surjection $f$ from $\omega$ onto $\eL_{\delta_\xi}$ as follows:  for $k\in\N$,
$z_\xi(k)=0 \iff ( k=2^m3^n \wedge f(m) \in f(n))$.

\item\label{clause.cof} $\mathcal G_\xi$ is the  compositional closure of $\{\sigma_\nu \colon \nu <\xi\}\cup\mathcal G_0$ in $S_\infty$.
\item\label{clause.ind}
$\mathcal G_\xi$ is cofinitary and $\mathcal G_\xi \in \eL_{\delta_\xi}$.
\item\label{clause.sigma} $\sigma_\xi = \sigma_{G}$, where $G$ is the unique $(\eL_{{\delta_\xi+\omega}}, \Q^{z_\xi}_{\mathcal G_\xi})$-generic 
obtained by hitting dense subsets of $\Q^{z_\xi}_{\mathcal G_\xi}$ in the order in which they are enumerated by the canonical surjection from $\omega$ onto $\eL_{\delta_\xi+\omega}$.
\item\label{clause.def}
Letting $\delta(\xi) = \big( \sup \{ \delta_\nu + \omega +\omega \colon \nu<\xi \}\big) +\omega$ for limit $\xi$ and $\delta(\xi)= \delta_{\xi-1} +\omega+\omega$ for successor $\xi<\omega_1$ we have 
\[
\langle (\delta_\nu, z_\nu, \mathcal G_\nu, \sigma_\nu) \colon \nu < \xi\rangle \in \eL_{\delta(\xi)}.
\]
\end{enumerate}
Obtaining such a sequence is straight-forward:
Having constructed $\vec{\mathcal G}\restriction \xi$,
\eqref{clause.delta} determines $\delta_\xi$  from $\langle \delta_\nu \colon \nu < \xi \rangle$ and \eqref{clause.z}  determines $z_\xi\in \eL_{\delta_\xi +\omega}$.
Note that $M_{z_\xi} = \eL_{\delta_\xi}$, and that by Fact \ref{f.project}, $\delta_\xi+\omega$ projects to $\omega$, as well.

Noting $\mathcal G_0 \in \eL_{\delta_0}$ for the case $\xi=0$, assume \eqref{clause.ind} by induction.  
Then \eqref{clause.sigma} uniquely determines $\sigma_\xi$ from ${\delta_\xi}$, $z_\xi$ and $\mathcal G_\xi$. 
That $\mathcal G_{\xi+1}$ is a cofinitary group follows by Lemma \ref{t.Q}.
Moreover, as the canonical surjection from $\omega$ to $\eL_{{\delta_\xi}+\omega}$ is in an element of 
$\eL_{{\delta_\xi}+\omega+\omega}$, so are $\sigma_\xi $ and $\mathcal G_{\xi+1}$. It follows that $\mathcal G_{\xi+1}\in\eL_{\delta_{\xi+1}}$ and the induction hypothesis \eqref{clause.ind} is verified for $\xi+1$.

It remains to verify \eqref{clause.ind} for limit $\xi < \omega_1$; for this, first show \eqref{clause.def}.
Observe that the inductive definition of $\vec{\mathcal G}$ by \eqref{clause.delta}, \eqref{clause.z} and \eqref{clause.sigma} involves only concepts which are absolute for initial segments of the $\eL$-hierarchy.
In fact, let $\Psi(x)$ be the formula that states that $x$ is the sequence obtained using said inductive definition restricted to ordinals in $\dom(x)$ and that there is no largest ordinal in the universe;
then $\vec g$ is equal to some initial segment of $\vec{\mathcal G}$ if and only if for any ordinal $\delta$ such that
$\vec g \in \eL_{\delta}$ we have $\eL_{\delta}\vDash \Psi(\vec g)$.
With this, \eqref{clause.def} follows by induction. 

In particular, it follows that $\mathcal G_\xi \in \eL_{\delta_\xi}$ for limit $\xi<\omega_1$, verifying \eqref{clause.ind}.
This finishes the inductive construction of $\vec{\mathcal G}$.

\medskip

We have just shown the following fact, crucial to the proof that $\mathcal G$ is $\Pi^1_1$:
\begin{fact}\label{c.Sigma_1}
There is a formula $\Psi(x)$ such that  $\vec g$ is equal to an initial segment 
$\langle (\mathcal G_\nu, \delta_\nu, z_\nu, \sigma_\nu) \colon \nu < \xi\rangle$ of $\vec{\mathcal G}$ if and only if
there is $\delta <\omega_1$ such that $\vec g \in \eL_\delta$ and $\eL_\delta\models \Psi(\vec g)$.
\end{fact}

Finally, we let
\[
\mathcal G= \bigcup_{\xi<\omega_1}\mathcal G_\xi,
\]
which is a  cofinitary group by \eqref{clause.ind} above.

To see that $\mathcal G$ is $\Pi^1_1$, we first show a weaker statement:

\begin{claim}\label{l.Sigma}
As a subset of $\N^\N$, $\mathcal G$ is $\Sigma^1_2$.
\end{claim}
\noindent
\emph{Proof of Claim.} By \cite{mathias-weak} and \cite{devlin-constructibility} (or by \cite{jensen-fine}), we may fix a sentence $\Lambda$ such that whenever $M$ is transitive and
$\langle M, \in\rangle\vDash \Lambda$, we have $M = \eL_\delta$ for some $\delta$.
By the Fact \ref{c.Sigma_1}, 
$\sigma\in \mathcal G$ if and only if 
\begin{multline}\label{eq.G.M}
\text{ there exists a countable transitive set $M$  s.t.\  $\langle M,\in\rangle \vDash\Lambda $ and}\\
\text{for some $\vec{g}\in M$ of the form }\vec{g} =\langle (\delta'_\xi, z'_\xi, \mathcal G'_\xi, \sigma'_\xi) \colon \xi \leq \nu \rangle \\
\text{we have $\sigma=w[\sigma'_\nu]$ for some $w \in W_{\mathcal G'_\nu,X}$ and $\langle M,\in\rangle \vDash \Psi(\vec{g})$.}
\end{multline}
By Fact \ref{f.truth},  \eqref{eq.G.M} is equivalent to a $\Sigma^1_2$ formula
\begin{equation}\label{eq.G.y}
 (\exists y\in 2^\N)\; \Phi(y,\sigma)
\end{equation}
where $\Phi(y,\sigma)$ is a $\Pi^1_1$ formula expressing
that
$E_y$ is well-founded and extensional and $M_y$ witnesses \eqref{eq.G.M}, i.e.\ $\langle \omega, E_y\rangle\vDash \Lambda$ and 
for some $n,m \in \omega$, $\pi_y(m) = \sigma$ and
$\langle \omega, E_y\rangle  \vDash\text{``} \Psi(n) \wedge
n =\langle (\delta'_\xi, z'_\xi, \mathcal G'_\xi, \sigma'_\xi) \colon \xi \leq \nu \rangle
\wedge \sigma'_\nu = m\text{''}$.

\hfill \qed{\tiny Claim \ref{l.Sigma}.}
\begin{lemma}\label{l.Pi}
In fact, $\mathcal G$ is $\Pi^1_1$ (as a subset of $\N^\N$).
\end{lemma}
\noindent
\emph{Proof of Lemma.}
The proof rests on the fact that $y$ as in \eqref{eq.G.y} can be found effectively in $\sigma$.

First fix some arbitrary $h^* \in \mathcal G_0\setminus\{\one\}$ (this is not necessary
but serves to better emphasize the structure of the argument) noting that by assumption $h^*$ is arithmetic.

For $\sigma \in \mathcal G\setminus \mathcal G_0$, 
by \eqref{clause.def} we may take $M$ in \eqref{eq.G.M} to be $\eL_{{\delta_\nu}+\omega+\omega}$ where $\nu$ is least such that $\sigma \in \mathcal G_{\nu+1}$.
Thus by Fact \ref{f.next}, $y$ in \eqref{eq.G.y} can be chosen so that $y \in \Delta^1_1(z_\nu)$.
We show that
$z_\nu$ is arithmetic in $\sigma$: 
Letting $\sigma = w[\sigma_G]$, this holds by construction if  $w$ has no conjugate proper subword.
Otherwise, $h^* w$ has no conjugate proper subword and so $z_\nu$ is computable in $h^* \sigma$ (any arithmetic element of $\mathcal G_\nu\setminus\{\one\}$ would do here).
In either case
 we obtain `$\Rightarrow$' in the following (`$\Leftarrow$' is obvious):
\[
\sigma\in \mathcal G \iff (\exists y \in \Delta^1_1(\sigma)) \; \Phi(y,\sigma).
\]
By Fact \ref{f.mansfield-solovay}, the right-hand side can be rendered as a $\Pi^1_1$ formula, proving the Lemma.
\hfill \qed{\tiny Lemma \ref{l.Pi}.}

\medskip

Since any $\tau \in S_\infty$ appears in some $\eL_{\delta_\xi}$, maximality of $\mathcal G$ follows from \eqref{t.Q.max} of Theorem \ref{t.Q}, and \eqref{clause.sigma} above.
In fact, we show the stronger statement:

\begin{lemma}\label{l.indestructible}
$\mathcal G$ is Cohen-indestructible.
\end{lemma}
\noindent
\emph{Proof of Lemma.}
Towards a contradiction, fix a $\C$-name $\dot \tau$ and $p\in \C$ such that
$p\forces_{\C} \langle \check{\mathcal G}, \dot \tau \rangle$ is cofinitary. 

For each $n\in \N$, pick
$\vec{p}_n=  \langle p^n_k \colon k \in \omega \rangle$ such that $\{ p^n_k \colon k \in \omega \}$ is a maximal antichain in $\C$
of conditions deciding $\dot \tau(n)$: i.e.\ we may find, for each $n\in\N$, a sequence $\vec{m}_n= \langle m^n_{k} \colon k \in \omega\rangle$ such that for each $k$, 
\[
p^n_k \forces \dot \tau (\check n) = \check m^n_k.
\]
Further, find $\xi<\omega_1$ such that for each $n\in\N$, $\{ \vec{p}_n, \vec{m}_n\}\subseteq  L_{{\delta_\xi}}$.
We may assume (by strengthening $p$ if necessary) that there is $N$ such that
\begin{equation}\label{e.cohen.forces}
p\forces_{\C} \lvert \{ n \in \N \colon \check\sigma_\xi(n) = \dot \tau(n) \} \rvert = \check N.
\end{equation}

By repeatedly using Lemma \ref{l.cohen.generic.hitting}, we may extend any given condition in $\Q^{z_\xi}_{\mathcal G_\xi}$ to a condition $q\in \Q^{z_\xi}_{\mathcal G_\xi}$ such that for some $p' \in \C$ stronger than $p$ and for some
set $Z\subseteq \dom(s^q)$ of size $N+1$ we have
\begin{equation*}\label{e.membership.in.D}
(\forall n\in Z) \; p'\forces_{\C} \dot \tau(\check n) = \check{s}^q(\check n).
\end{equation*}
Just as in Lemma \ref{t.Q}, we must circumvent the use of the forcing relation: 
Since each $\vec{p}_n$ enumerates a maximal antichain, by extending $p'$ finitely many times, we can assume that for every $n\in Z$, $p' \leq_{\C}p^n_k$ for some $k$ such that $m^n_k = s^q(n)$.

Let $D$ be the set of $q\in \Q^{z_\xi}_{\mathcal G_\xi}$ such that for some $p' \in \C$ stronger than $p$ and for some
set $Z\subseteq \dom(s^q)$ of size $N+1$ we have
that for every $n\in Z$, $p'\leq_{\C}p^n_k$ for some $k$ such that $m^n_k = s^q(n)$.

We have just shown that $D$ is dense $\Q^{z_\xi}_{\mathcal G_\xi}$; as $D\in \eL_{\delta_\xi+\omega}$,
the generic which gave rise to $\sigma_\xi$ meets $D$ and we conclude
that for some 
$p' \in \C$ stronger than $p$ and for some
set $Z\subseteq\N$ of size $N+1$ we have
\[
(\forall n\in Z) \; p'\forces_{\C} \dot \tau(\check n) = \check{\sigma}_\xi(\check n),
\]
contradicting \eqref{e.cohen.forces}; thus, $\mathcal G$ is Cohen-indestructible.
\hfill \qed{\tiny Lemma \ref{l.indestructible}.}
\end{proof}
As an immediate corollary, we obtain (a strong form of) Theorem  \ref{t.continuum}:
\begin{corollary}\label{c.main}
The existence of a constructible $\Pi^1_1$ maximal cofinitary group of size $\omega_1$ is consistent with  arbitrarily large continuum (relative to \textup{ZFC}).
\end{corollary}
\begin{proof}
Note that if $\Psi(x)$ is the $\Pi^1_1$ formula which we constructed in Claim~\ref{l.Pi} and which defines membership in $\mathcal G$,
then the following holds in $\Vee$:  
\[
(\forall \sigma \in \N^\N) \; \Psi(\sigma) \iff \sigma \in \eL \wedge \Psi(\sigma)^{\eL}.
\]
This is because $\Psi(\sigma)$ is easily seen to imply $\sigma\in \eL$.
\end{proof}

\section{Questions}\label{s.questions}

Considering the many known models where some inequality holds between $\mathfrak{a}_g$ and
another  cardinal invariant of the continuum,
the methods developed in the present paper suggest to consider the following definable analogue:

For $\Gamma$ an arbitrary pointclass,
let $\mathfrak{a}_g(\Gamma)$ be the least cardinal $\kappa$ such that there is a mcg $\mathcal G\in \Gamma$ of size $\kappa$; if there is no mcg $\mathcal G\in\Gamma$, let $\mathfrak{a}_g(\Gamma)=\infty$.

\begin{question}
Which inequalities involving $\mathfrak{a}_g({\Pi^1_1})$ and other cardinal invariants of the continuum can be shown to hold consistently?
\end{question}
An earlier version of this paper asked whether there can be a Borel mcg; as mentioned in \S\ref{intro} this question has since been answered in the affirmative (see \cite{borel-mcg}, in which the authors also state as their goal to show in a future paper that a closed mcg exists; also compare \cite{on-shelah} which shows that there is a closed so-called \emph{maximal eventually different family}).

\bibliographystyle{amsplain}
\bibliography{mcg}

\end{document}